\documentclass[12pt, letterpaper]{article}
\usepackage{amsmath}
\usepackage{amsthm}
\usepackage{amsfonts}
\usepackage{amssymb}
\usepackage{color}
\usepackage{enumitem}
\newtheorem{theorem}{Theorem}[section]
\newtheorem{corollary}{Corollary}[section]
\newtheorem{lemma}{Lemma}[subsection]
\newtheorem{definition}{Definition}[section]
\newtheorem{proposition}{Proposition}[subsection]
\newtheorem{example}{Example}[subsection]
\newtheorem{remark}{Remark}[section]

\allowdisplaybreaks

\begin{document}
\title{Multi-marginal optimal transportation problem for cyclic costs\thanks{BP is pleased
		to acknowledge support from Natural Sciences and Engineering Research
		Council of Canada Grant 04658-2018. The work of AVJ was completed in partial fulfillment of the requirements for a doctoral degree in mathematics at the University of Alberta.}}
\author{BRENDAN PASS \thanks{%
		Department of Mathematical and Statistical Sciences, University of Alberta,
		Edmonton, Alberta, Canada pass@ualberta.ca.}\\	 
	\and 
	ADOLFO VARGAS-JIM\'{E}NEZ\thanks{%
		Department of Mathematical and Statistical Sciences, University of Alberta,
		Edmonton, Alberta, Canada vargasji@ualberta.ca.}}
\date{\today}
\maketitle
\begin{abstract}
We study a multi-marginal optimal transportation problem with a cost function of the form $c(x_{1}, \ldots,x_{m})=\sum_{k=1}^{m-1}|x_{k}-x_{k+1}|^{2} + |x_{m}- F(x_{1})|^{2}$, where $F: \mathbb{R}^n \rightarrow \mathbb{R}^n$. 
 When $m=4$, $F$ is a positive multiple of the identity mapping, and the first and last marginals are absolutely continuous with respect to Lebesgue measure, we establish that any solution of the Kantorovich problem is induced by a map; the solution is therefore unique.  We go on to show that this result is sharp in a certain sense.  Precisely, we exhibit examples showing that Kantorovich solutions may concentrate on a higher dimensional sets if any of the following hold: 1) $F$ is any linear mapping other than a positive scalar multiple of the identity, 2) the last marginal is not absolutely continuous with respect to Lebesgue measure, or 3) the number of marginals	$m \geq 5$, even when $F$ is the identity mapping.
	
\end{abstract}

\section{INTRODUCTION}

Let $\mu_{k}\in P(X_{k})$ be Borel probability measures, where $X_{k}\subseteq \mathbb{R}^{n}$ are bounded domains and $k=1,\ldots, m$. For a real-valued cost function $c$ on the product space $X := X_{1}\times X_{2}\times \ldots \times X_{m}$, the multi-marginal version of Kantorovich's optimal transportation problem is to solve:
\begin{equation}\label{KP}
   \inf \bigg \{\mathcal{F}_{c}[\mu]: \mu \in \Pi(\mu_{1},\ldots , \mu_{m})\bigg\} \tag{KP}
  \end{equation}
 where
 \begin{flalign*}
\Pi(\mu_{1},\ldots , \mu_{m})&= \Big\{  \mu \in P(X): \mu(X_{1}\times \ldots \times A_{k} \times \ldots \times X_{m})=\mu_{k} (A_{k})\\
 &\qquad\quad\qquad\qquad\text{ for all measurable sets } A_{k}\subseteq X_{k},\quad 1\leq k\leq m \Big\} 
 \end{flalign*}
 and $\mathcal{F}_{c}[\mu]=\displaystyle \int_{X} c(x_{1},\ldots, x_{m})d\mu.$\\
 The Monge formulation of the same problem is to solve:
 \begin{equation}\label{MP}
   \inf \bigg \{\mathcal{G}[(T_{1},T_{2},\ldots, T_{m})]: \{T_{k}\}_{k=1}^{m}\in\Gamma(\mu_{1},\ldots , \mu_{m})\bigg\}\tag{MP}
  \end{equation}
where 
\begin{equation*}
\Gamma(\mu_{1},\ldots , \mu_{m})= \Big\{  \{T_{k}\}_{k=1}^{m}: T_{k}: X_{1}\mapsto X_{k}\quad \textit{is measurable},\quad (T_{k})_{\sharp} \mu_{1} = \mu_{k} \quad \textit{and}\quad T_{1}=Id\Big\}
\end{equation*}
and $\mathcal{G}[(T_{1},T_{2},\ldots, T_{m})]=\displaystyle \int_{X_{1}} c(x_{1},T_{2}x_{1},\ldots, T_{m}x_{1})d\mu_{1}.$

 As usual, $Id$ denotes the identity map and $(T_{k})_{\sharp} \mu_{1}$ the image measure of $\mu_{1}$ through $T_{k}$ given by $(T_{k})_{\sharp} \mu_{1}(A_{k}) = \mu_{1}(T_{k}^{-1}(A_{k}))$, for every measurable subset $A_{k}$ of $X_{k}$.  For any $\{T_{k}\}_{k=1}^{m}\in\Gamma(\mu_{1},\ldots , \mu_{m})$, we can define the measure $\mu =(T_{1}, T_{2},\ldots, T_{m})_{\sharp} \mu_{1} \in \Pi(\mu_{1},\ldots , \mu_{m})$. Then $\mathcal{F}_{c}[\mu]=\mathcal{G}[(T_{1},T_{2},\ldots, T_{m})]$; thus, (KP) can be described as a relaxed version of (MP).
When $m=2$, these reduce respectively to the Kantorovich and Monge formulations of the classical optimal transport problem, which has profound connections with many different areas of mathematics, and an extremely wide range of applications in other fields, surveyed in, for example, \cite{Filippo}\cite{Villani} and \cite{Villani2}.  The multi-marginal generalization (ie, the case $m \geq 3$) has attracted a great deal of attention recently, largely due to its own substantial collection of  applications, in, for instance, physics, economics, finance and image processing (see \cite {Pass3} for an overview).

Under very general conditions (for instance, compactness of the spaces and continuity of the cost is more than enough) there exists a solution for (KP). In the two marginal case, a simple condition on $c$, known as the \emph{twist} condition, as well as absolute continuity of the first marginal $\mu_1$, is well known to ensure existence of Monge solutions, and, consequently, uniqueness in both problems \cite{Gangbo1}\cite{Gangbo2}.

 The situation is much more subtle for $m\geq 3$. Beginning with a seminal work by Gangbo and \'{S}wi\c{e}ch on the cost $c(x_{1}, \ldots,x_{m})=\sum_{k\neq i}^{m}|x_{k}-x_{i}|^{2}$ \cite{Gangbo}, many authors have proven that several particular cost functions yield Monge solutions \cite{Heinich}\cite{Carlier03}\cite{Pass4}, while a variety of others do not \cite{Carlier2}\cite{Pass1}\cite{Pass13}\cite{ColomboDePascaleDiMarino15}. An important challenge in current research is to fully identity and classify the  costs which ensure unique Monge solutions.

A sufficient condition for Monge solutions (known as the \emph{twist on splitting sets}), which encompasses all known examples of costs leading to Monge solutions, is provided in \cite{Kim}.  However, this condition is much more complicated and difficult to check for a given cost than its two marginal counterpart. Second order differential conditions on $c$ which guarantee the twist on splitting sets are known, but they are very strong; of  particular relevance to us, they require invertability of certain matrices of second derivatives which does not hold in some situations of interest \cite{Pass0}\cite{Pass4}. 

In this paper, we study the cost 
\begin{equation}\label{eqn: cyclic cost}
c(x_{1}, \ldots,x_{m})=\sum_{i=1}^{m-1}|x_i-x_{i+1}|^2 +|x_m-F(x_1)|^2,
\end{equation}
 for a fixed map $F: \mathbb{R}^n \rightarrow \mathbb{R}^n$. As highlighted in Section 1.7.4 of \cite{Filippo}, this cost measures the discrete time kinetic energy of a cloud of particles whose density at timestep $k$ is $\mu_k$, such that the final position of the particle initially at $x_1$ is fixed to be $F(x_1)$. In particular, when each $\mu_{k}=\mathcal{L}^{n}|_D$ is (normalized) Lebesgue measure on a common bounded domain $X_k=D \subset \mathbb{R}^n$ and $F: D \rightarrow  D$ is measure preserving, $F_\#\mu_k =\mu_k$, the Monge problem with this cost corresponds to the time discretization of Arnold's variational interpretation of the incompressible Euler equation \cite{Arnold66}; the Kantorovich formulation corresponds to a discretization of Brenier's generalization \cite{Brenier}.   If $m=2$ and $I+DF(x)$ is invertible (alternatively it corresponds to the quadratic cost up to a change of variables) where $I$ denotes the identity matrix, the cost is twisted; while for $m=3$, it is twisted on splitting sets as long as 
$DF(x) +DF(x)^T>0$. To the best of our knowledge, little is known analytically about the structure of solutions for $m \geq 4$, though the problem has received a fair bit of attention from a numerical perspective \cite{Thomas}\cite{Mirebeau}\cite{Benamou}\cite{Brenier08}.

As we show, cost function \eqref{eqn: cyclic cost} is not twisted on splitting sets for $m\geq 4$.  Nevertheless,  when $m=4$ and $F$ is a positive scalar multiple of the identity mapping, we are able to prove that all solutions are of Monge type, and therefore unique, under an additional regularity condition on the marginals (in addition to $\mu_1$, either $\mu_2$ or $\mu_4$ must be absolutely continuous).   This result is very special; indeed, as we show later on, it is in some sense impossible to go further.  A simple example shows that the extra regularity condition on $\mu_4$ or $\mu_2$ is required.  When $m=4$, and $F$ is any linear mapping other than a positive scalar multiple of the identity, we demonstrate that solutions may not be of Monge type, even for diffuse marginals.  Similarly, when  $m \geq 5$, we prove that solutions may not be of Monge type, even for $F(x)=x$.

To offer some perspective on these results, we note that generalized incompressible flows (ie, solutions to the infinite marginal version of the Kantorovich problem, when each marginal is uniform and $F$ measure preserving) are not generally unique  in dimension $n\geq 2$ \cite{BernotFigalliSantambrogio08}; however, unique Monge-type solutions exist when $F$ is close to the identity mapping \cite{EbinMarsden70}.  It seems reasonable to expect the same to hold for the time discretized problem. Our counterexamples essentially show that this is not the case for $m\geq 4$, at least when the marginals are allowed to differ.

 In the next section, we recall some preliminary results and definitions which will be needed in the paper.  In Section 3, we prove that solutions are of Monge type when the number of marginals $m=4$, $F(x)=x$ and the first and last marginal are absolutely continuous.  The fourth section is reserved for examples of non-Monge solutions when these conditions are violated.
 \section{\textbf{DEFINITIONS AND PRELIMINARIES
 }}
For the cost function \eqref{eqn: cyclic cost} 
 on $X = X_{1}\times X_{2}\times \ldots \times X_{m}$, we will approach the problem of minimizing $\mathcal{F}_{c}[\mu]=\displaystyle \int_{X} c(x_{1}, \ldots,x_{m})d\mu$, by the equivalent problem of maximizing: 

\begin{equation}\label{KPb}
  \mathcal{F}_{b}[\mu]=\displaystyle \int_{X} b(x_{1},\ldots, x_{m})d\mu \tag{KPb}
  \end{equation}
 where $b(x_{1},\ldots, x_{m})=\sum_{k=1}^{m-1}x_{k}\cdot x_{k+1}+x_{m}\cdot F(x_{1})$, over $\mu \in \Pi(\mu_{1},\ldots , \mu_{m})$.
 
 The dual of \eqref{KPb} is to minimize:
 
\begin{equation}\label{DP}
\sum_{k=1}^{m}\int_{X_{k}}u_{k}(x_{k})d\mu_{k}\tag{DP}
  \end{equation}
among all m-tuples $(u_{1}, u_{2},\ldots, u_{m})$ where each $u_{k} \in L^{1}(\mu_{k})$ and  $\sum_{k=1}^{m}u_{k} \geq b(x_{1},\ldots,x_{m})$, for all $(x_{1},\ldots,x_{m})\in X$.

We now introduce an important class of functions satisfying the constraint in \eqref{DP}.
\begin{definition}
An $m$-tuple of functions $(u_{1}, u_{2},\ldots, u_{m})$ is $b$-conjugate if for all k $$u_{k}(x_{k})= sup_{x_i\in X_{i}, i\neq k}\Big ( b(x_{1},\ldots,x_{m}) -\sum_{i\neq k}u_{i}(x_{i})\Big ) $$
\end{definition}
It is well known that if  $(u_{1}, u_{2},\ldots, u_{m})$ is $b$-conjugate, then each $u_{k}$ in inherits local Lipschitz and semi-convexity properties from $b$ \cite{McCann}.

The following well known duality result captures the connection between \eqref{DP} and \eqref{KPb}. Most of the assertions can be traced back to Kellerer \cite{Kellerer}; a proof of the $b$-conjugacy of the solutions can be found in \cite{Gangbo} or \cite{Pass4}.
\begin{theorem}
Assume $X_{k}$ is compact for every $k$. Then, there exists a solution $\mu$ to the Kantorovich problem and a $b$-conjugate solution $(u_{1}, u_{2},\ldots,u_{m})$ to its dual. The maximum and minimum values in \eqref{DP} and \eqref{KPb} respectively are the same and  $\sum_{k=1}^{m}u_{k}(x_{k})=  b(x_{1},\ldots,x_{m})$ for all $(x_{1},\ldots,x_{m})\in spt(\mu)$, where $spt(\mu)$ denotes the support of $\mu$.
\end{theorem}


 We will use the next theorem, established in  \cite{Ballantine}\cite{Ballantine2}, in the construction of counterexamples. For this purpose,  we denote by $\Re^{d}$ the set of all $2\times 2$ real matrices that can be expressed as the product of $d$ positive definite real matrices.

\begin{theorem}
 Assume that $M=\begin{pmatrix}
a& b\\
c& d
\end{pmatrix}
$ is a $2\times 2$ matrix and $\vert M\vert >0$, where $\vert M\vert$ denotes the determinant of $M$, then:
\begin{enumerate}
\item $M\in \Re^{2}$ iff $M$ is diagonalizable and its eigenvalues are both positive.
\item $M\in \Re^{3}$ iff $tr(M)> 0$ or $(c-b)^{2}>4|M|$.
\end{enumerate} 
\end{theorem}
We close this section by recalling a couple of well known formulas which will be useful in the construction of counterexamples in Section \ref{sect: examples}.

 For any $2\times 2$ matrices $A$ and $B$ we have:
\begin{equation}\label{eqn: det of sum}
\vert A + B\vert = \vert A\vert + \vert B\vert + trace\big(Adj(A)B\big)
\end{equation}
 where $Adj(A)$ denote the adjugate of $A$.
 
Given a convex function $f: \mathbb{R}^{n}\longrightarrow \mathbb{R}\cup \{\infty\}$, its Legendre-transform will be denoted by $f^{*}$; that is, $f^{*}(y)= \sup_{x} (x\cdot y - f(x))$. We have special interest in the Legendre-transform of $f(x)=\frac{1}{2}x^{T}Ax + b\cdot x$ for a given positive definite $n\times n$ matrix $A$ and $b \in \mathbb{R}^n$. For this function, we have:
\begin{equation}\label{Lf}
f^{*}(y)=\frac{1}{2}(y-b)^{T}A^{-1}(y-b).
\end{equation}
\section{\textbf{MONGE SOLUTIONS
}}
We now show that under regularity conditions on the first and fourth margi-nal, we obtain a unique Monge solution for the case $m=4$ and $F(x)=x$.\\
In what follows $\mathcal{L}^{n}$ denotes the Lebesgue measure on $\mathbb{R}^{n}$.
\begin{theorem}
Let $\mu_{k}$ be probability measures over open bounded sets $X_{k}\subseteq\mathbb{R}^{n}$, $k=1, 2, 3, 4$. Take $b(x,y,z,w)= x\cdot y + y\cdot z + z\cdot w + w\cdot x $ and assume $\mu_{1}, \mu_{4}$ are absolutely continuous with respect to $\mathcal{L}^{n}$. Then any solution of the Kantorovich problem \eqref{KPb} is induced by a map.
\end{theorem}
\begin{proof}
Let $\mu$ be a solution to \eqref{KPb} and $(u_{1}, u_{2}, u_{3}, u_{4})$ a $b$-conjugate solution to its dual. Consider the set
 $$S=\Big\{(x,y,z,w): Du_{1}(x)\;\text{and } Du_{4}(w)\;\text{exist and} \; b(x,y,z,w)= u_{1}(x)+ u_{2}(y)+ u_{3}(z)+ u_{4}(w)\Big\}.$$
Since the functions $u_{1}(x)$ and $u_{4}(w)$ are Lipschitz, they are differentiable $\mathcal{L}^{n}$-a.e., and therefore $\mu_{1}$ and $\mu_{4}$  a.e. by hypothesis. Hence, $\mu(S)=1$. By setting $f(x,y)=x\cdot y + sup_{z} [y\cdot z-u_{3}(z) + h(x+z)]$ where $ h(x+z)= sup_{w}[(x+z)\cdot w - u_{4}(w)]$, we obtain: $$b(x,y,z,w)-u_{3}(z)- u_{4}(w)\leq f(x,y)\leq u_{1}(x)+ u_{2}(y)$$ for all $x,y,z,w$ and in particular equality holds on  $S$. 

Now, for any   fixed $x_{0}$, we will show that there is only one $y_{0},z_{0},w_{0}$ such that in $(x_0, y_{0},z_{0},w_{0}) \in S$. 
Since the function $x\mapsto f(x,y_{0})$ is convex and $f(x,y_{0})\leqslant u_{1}(x)+u_{2}(y_{0})$ for every $x$, it is subdifferentiable everywhere. For $(x_0, y_{0},z_{0},w_{0}) \in S$ the equality $f(x_{0},y_{0})=u_{1}(x_{0})+u_{2}(y_{0})$ implies that the subdifferential of $f(x,y_{0})$ at $x_{0}$ is contained in the subdifferential of $u_{1}(x)$ at $x_{0}$, which is $\{ Du_{1}(x_{0})\}$; that is, $D_{x}f(x_{0},y_{0})$ exists and equals $Du_{1}(x_{0})$. By a similar argument  $Dh(x_{0}+ z_{0})$ exists, $Dh(x_{0}+ z_{0}) =D_{x}f(x_{0},y_{0})-y_0=w_0$, and clearly, $z_{0} \in \text{argmax}[ y_{0}\cdot z - u_{3}(z) + h(x_{0}+z)]$. 
We claim that the map $(y,z,w)\mapsto D_{x}f(x_{0},y)$ with domain $R:=\{ (y,z,w): (x_{0},y,z,w) \in S\}$ is injective; this will imply the desired result.

Assume $D_{x}f(x_{0},y_{1}) = D_{x}f(x_{0},y_{2})$ for some $(y_{1},z_{1},w_{1}), (y_{2},z_{2},w_{2})\in R$.  Note that 
\begin{equation}\label{eqn: first order h}
y_1+w_1=y_{1} + Dh(x_{0}+ z_{1})=D_{x}f(x_{0},y_{1}) = D_{x}f(x_{0},y_{2})=y_{2} + Dh(x_{0}+ z_{2})=y_2+w_2
\end{equation}
 and $z_{k} \in \text{argmax}[ y_{k}\cdot z - u_{3}(z) + h(x_{0}+z)]$, $k =1,2$. 
 Therefore 
\begin{eqnarray}\label{eqn: y1z1 inequality}
 y_{1}\cdot z_{2} - u_{3}(z_{2}) + h(x_{0}+z_{2})&\leq&  y_{1}\cdot z_{1} - u_{3}(z_{1}) + h(x_{0}+z_{1})\\
 y_{2}\cdot z_{1} - u_{3}(z_{1}) + h(x_{0}+z_{1})&\leq & y_{2}\cdot z_{2} - u_{3}(z_{2}) + h(x_{0}+z_{2});\label{eqn: y2z2 inequality}
 \end{eqnarray}
adding these inequalities gives $(y_{1}-y_{2})\cdot(z_{2}-z_{1})\leq 0$, then by \eqref{eqn: first order h},
\begin{equation}\label{eqn: wz dot product negative}
(w_{2}-w_{1})\cdot(z_{2}-z_{1})\leq 0.
\end{equation}
Furthermore, since $w_k  \in \text{argmax}[ (x_{0}+z_{k})\cdot w - u_{4}(w)] $,
 \begin{eqnarray} \label{eqn: xz1w1 inequality}
 (x_{0}+z_{1})\cdot w_{2} - u_{4}(w_{2})&\leq& h(x_{0}+z_{1})=(x_{0}+z_{1})\cdot w_{1} - u_{4}(w_{1})\\
 (x_{0}+z_{2})\cdot w_{1} - u_{4}(w_{1})&\leq& h(x_{0}+z_{2})=(x_{0}+z_{2})\cdot w_{2} - u_{4}(w_{2});\label{eqn: xz2w2 inequality}
 \end{eqnarray}
after adding and canceling similar terms we obtain 
 \begin{equation} \label{eqn: wz dot product positive}
 (w_{2}-w_{1})\cdot (z_{1}-z_{2})\leq 0.
 \end{equation}
Therefore, by \eqref{eqn: wz dot product negative} and \eqref{eqn: wz dot product positive}, $(w_{2}-w_{1})\cdot (z_{1}-z_{2})=0$ and we must have equality in \eqref{eqn: y1z1 inequality}, \eqref{eqn: y2z2 inequality}, \eqref{eqn: xz1w1 inequality} and \eqref{eqn: xz2w2 inequality}.   This implies that   $w_{2}\in \text{argmax}[ (x_{0}+z_{2})\cdot w - u_{4}(w)] \bigcap  \text{argmax}[ (x_{0}+z_{1})\cdot w - u_{4}(w)]$; additionally, $(y_{2},z_{2},w_{2})\in R$ implies $u_{4}(w)$ is differentiable at $w_{2}$, and so 
\begin{equation}
x_{0}+z_{1}=Du_{4}(w_{2})=x_{0}+z_{2};
\end{equation}
that is, $z_1=z_2$.  The equality $w_k =Dh(x_0+z_k)$ for $k=1,2$ then implies that $w_{1}=w_{2}$, and so $y_{1}=y_{2}$ by \eqref{eqn: first order h}.

In summary, the equation $D_{x}f(x_{0},y_{0})=Du_{1}(x_{0})$, which holds on $S$ and therefore $\mu$ almost everywhere, implies that $(y_{0},z_{0},w_{0})$ is uniquely defined from $x_{0}$; therefore, the 3-tuple $(T_{2},T_{3},T_{4})$ where $T_{k}$ is the map associating each $x_{0}$ to $y_{0}$, $z_{0}$ and $w_{0}$ respectively, induces $\mu$. 
\end{proof}
\begin{remark}
In the proof above, $T_{2}$ solves a sort of  \emph{effective}, two marginal optimal transportation problem with surplus $f$ and marginals $\mu_{1}$  and  $\mu_{2}$. Our argument essentially verifies that $f$ is twisted; that is, $y\mapsto D_{x}f(x_{0},y)$ is injective (this condition is well known to ensure Monge solution for two marginal problems \cite{Filippo}). 
\end{remark}
\begin{remark}
In a similar way, we can prove Theorem 3.1 if we replace $F=I$ by $F=\lambda I$, where $\lambda >0 $ is a scalar.  
\end{remark}
A standard argument now implies uniqueness of solutions to \eqref{KPb}.
\begin{corollary}
Assume the same conditions as Theorem 3.1. Then the solution to the Kantorovich problem \eqref{KPb} is unique.
\end{corollary}
\begin{proof}
Let $\mu^{1}$ and $\mu^{2}$ be distinct solutions of \eqref{KPb}. By Theorem 3.1,
$\mu^{1}=(Id,T_{2}^{1},T_{3}^{1},T_{4}^{1})$ and $\mu^{2}=(Id,T_{2}^{2},T_{3}^{2},T_{4}^{2})$ for some 3-tuples of measurable maps $(T_{2}^{1},T_{3}^{1},T_{4}^{1}) \neq (T_{2}^{2},T_{3}^{2},T_{4}^{2})$. Since the set of solutions of \eqref{KPb} is convex, $\mu=\frac{1}{2}\mu^{1}+\frac{1}{2}\mu^{2}$  is also a solution. Hence, applying  one more time Theorem 3.1, we conclude that $\mu$ is concentrated on a graph. This is clearly not possible, completing the proof.
\end{proof}
\section{\textbf{NON-MONGE SOLUTIONS.
 }}\label{sect: examples}
We now illustrate why the conditions on the marginal $\mu_{4}$, the number of variables $m$ and  the map $F$ in the definition of $b$ and $c$ of Theorem 3.1 are necessary.
\subsection{The regularity condition on $\mu_{4}$.}
Assuming $m$ and $F$ as in Theorem 3.1 the next example will show that if $\mu_{4}$ is not absolutely continuous, we can find a solution for \eqref{KP} not induced by a map. Furthermore, the uniqueness result of Corollary 3.1 fails.
\begin{example}  Let $X_{k}=B(0,r)\subseteq\mathbb{R}^{n}$ be an open ball, $r>0$. Consider $c(x,y,z,w) = \frac{1}{2}\big(|x-y|^{2}+|y-z|^{2}+|z-w|^{2}+|w-x|^{2}\big)$ and the following measures on $X_{k}$: The Dirac measure at the origin $\mu_{2}=\mu_{4}=\delta_{0}$ and the renormalized n-dimensional Lebesgue measure $\mu_{1}=\mu_{3}=\frac{\mathcal{L}^{n}}{k_{n}r^{n}}$, where $k_{n}$ is the volume of the n-dimensional ball of radius 1. Take any $\mu$ in $\Pi(\mu_{1},\mu_{2},\mu_{3},\mu_{4})$. Since $(x,y,z,w)\in spt(\mu)$ implies  $y=w=0$, we obtain:
\begin{flalign*}
\displaystyle \int_{X_{1}\times X_{2}\times X_{3}\times X_{4}} c(x,y,z,w) d\mu & =\displaystyle \int_{X_{1}\times X_{2}\times X_{3}\times X_{4}} \Big(|x|^{2}+|z|^{2}\Big) d\mu \\
  &=\displaystyle \int_{B(0,r)} |x|^{2} d\mu_{1}(x) + \int_{B(0,r)} |z|^{2} d\mu_{3}(z);\\
\end{flalign*} 
that is, $\mathcal{F}[\mu]$ is independent of $\mu$, hence any element in $\Pi(\mu_{1},\mu_{2},\mu_{3},\mu_{4})$ is a minimizer. Therefore we can find optimal measures  $\mu$ for (KP) not concentrated on a graph of a measurable map; for instance, the product measure $\mu=\mu_{1}\otimes\mu_{2}\otimes\mu_{3}\otimes\mu_{4}$. On the other hand, if we set $\mu=(Id, F, T, F)_{\sharp}\mu_{1}$ where $T_{\sharp}\mu_{1}=\mu_{3}$ and $F=0$, we get solutions for the Monge problem.
\end{example}  
 \subsection{The condition $F= I$.}
In this subsection, by assuming $F$ is not a positive multiple of the identity mapping, $m=4$ and $n=2$, we will find absolutely continuous marginals in $\mathbb{R}^{2}$ such that a solution of \eqref{KPb} is concentrated in a 3-dimensional set. Therefore, this solution will not be induced by a map.  For this purpose, the following lemma will be needed:
\begin{lemma}  For each $2\times 2$ real matrix $F$ such that $ 
F \neq \lambda I$ for some $\lambda > 0$,
 there exists
 $M 
\in\Re^{2}$ such that $F+M$ is singular.
\end{lemma}
\begin{proof}
Let  $F =\begin{pmatrix}
a& b\\
c& d
\end{pmatrix}
$ be $2\times 2$ real matrix such that $F \neq \lambda I$ for any $\lambda >0$. We want to show that $\vert F + M\vert =0$ for some $M =\begin{pmatrix}
e& f\\
g& h
\end{pmatrix}\in\Re^{2}$. First, note that by \eqref{eqn: det of sum}

$$\vert F + M\vert = \vert F\vert + \vert M\vert + (de-gb) +(ah-fc).$$ \\
We divide the proof into 3 cases:
\begin{enumerate}
\item If $c\neq 0$, take any $e,h\in \mathbb{R}$ with $e\neq h$ and $e,h> 0$. By setting $f= \frac{\vert F\vert + eh + de + ah}{c}$ and  $g=0$ we obtain $\vert M\vert >0$ and  $\vert F + M\vert = 0$. Furthermore, $M$  is triangular with distinct eigenvalues $e$ and $h$, hence $M$ is diagonalizable. Since $e,h> 0$, we get $M\in\Re^{2}$ by Theorem 2.2.
\item  If $c= 0$ and $b\neq 0$, using a similar argument as in 1. we obtain the same result by taking $f=0$, $g= \frac{\vert F\vert + eh + de + ah}{b}$ and any $e,h> 0$, $e\neq h$.
\item If $b=c=0$, note that by hypothesis $a\neq d$. Also, we have $a,d \geq 0$, $a,d<0$ or without loss of generality $a< 0$ and $d\geq 0$. For the second and third case, we can make $M$ diagonalizable with positive determinant and satisfying $\vert F + M\vert = 0$, by taking $e=h=-a$ and $f=g=0$. Hence, $M\in\Re^{2}$ by Theorem 2.2. For the first case assume without loss of generality $a>d \geq 0$ and consider the matrix:

$$M =\begin{pmatrix}
e& f\\
g& h
\end{pmatrix}=\begin{pmatrix}
\frac{ad}{a-d}+\lambda& \frac{d^{2}}{a-d}+\frac{a+d}{2a}\lambda\\
-\frac{a^{2}}{a-d}-\lambda& \frac{-ad}{a-d}-\frac{a+d}{2a}\lambda\\
\end{pmatrix}$$

with $\lambda>0$. Clearly, $$\vert F+M \vert=\begin{vmatrix}
\frac{a^{2}}{a-d}+\lambda& \frac{d^{2}}{a-d}+\frac{a+d}{2a}\lambda\\
-\frac{a^{2}}{a-d}-\lambda& -\frac{d^{2}}{a-d}-\frac{a+d}{2a}\lambda\\
\end{vmatrix}=0.$$
Since 
$\frac{d^{2}}{a-d}+\frac{a+d}{2a}\lambda=-d + \frac{ad}{a-d}+\frac{a+d}{2a}\lambda \quad \text{and} \quad\frac{ad}{a-d}+\lambda= -a + \frac{a^{2}}{a-d}+\lambda$ we have:
\begin{flalign*}
 \vert M\vert & =-\Big(-a + \frac{a^{2}}{a-d}+\lambda \Big)\Big(\frac{ad}{a-d}+\frac{a+d}{2a}\lambda\Big) + \Big(\frac{a^{2}}{a-d}+\lambda\Big)\Big(-d + \frac{ad}{a-d}+\frac{a+d}{2a}\lambda\Big)&\nonumber\\
    &=a\Big(\frac{ad}{a-d}+\frac{a+d}{2a}\lambda\Big)-d\Big(\frac{a^{2}}{a-d}+\lambda\Big)\nonumber\\
     &=\Big(\frac{a+d}{2}-d\Big)\lambda
     \nonumber\\
      &=\Big(\frac{a-d}{2}\Big)\lambda>0.
       \nonumber\\
\end{flalign*} 

Furthermore, $tr(M)=e+h=\frac{a-d}{2a}\lambda>0$. Hence, $ tr(M)^{2}-4|M|>0$  for big enough $\lambda$; that is, the eigenvalues of $M$ given by $$\dfrac{tr(M)\pm \sqrt{tr(M)^{2}-4|M|}}{2} \quad$$ 
 are both positive and different. Then $M$ is diagonalizable and belongs to $\Re^{2}$, by Theorem 2.2.
\end{enumerate}
\end{proof}
\begin{proposition}\label{prop:4 marginal counterexample} \textit{For $b(x,y,z,w)=x\cdot y + y\cdot z+z\cdot w+w\cdot F(x)$, with $(x,y,z,w)\in (\mathbb{R}^{2})^{4}$  and $F$ a linear map such that $F\neq \lambda I$ for any $\lambda >0$}, there are absolutely continuous marginals $\mu_{1}, \mu_{2}, \mu_{3}, \mu_{4}$ with respect to $\mathcal{L}^{2}$, such that a solution of \eqref{KPb} is not concentrated on a graph of a measurable map.
\end{proposition}
\begin{proof}
Let  $F =\begin{pmatrix}
a& b\\
c& d
\end{pmatrix}
$ be the matrix representation of $F(x)$. By the previous lemma we can choose $M =\begin{pmatrix}
e& f\\
g& h
\end{pmatrix}
\in\Re^{2}$ such that $F+M$ is singular. Note that $A:=M^{-1}F+I$ is also singular and $M^{-1}\in \Re^{2}$. Let $M_{1}, M_{2}>0$ be such that $M^{-1}=M_{1}M_{2}$ and $v\in \mathbb{R}^{2}$ a nonzero vector satisfying $Av=0$. Decompose each vector $x\in \mathbb{R}^{2}$ into orthogonal components $x=x_{\perp} + x_{\parallel}$ with $x_{\perp} \perp v$ and $x_{\parallel} \parallel v$. For all $x,y,z,w$ define:
\begin{enumerate}[label=\roman*.]
\item  $u_{1}(x)=\frac{|x_{\perp}|^{2}}{2}+g_{1}(x) + g_{2}(x), \quad \quad u_{2}(y)= \frac{|A^{T}y|^{2}}{2}+g(y),\quad$ where \smallskip

$g_{1}(x)=\frac{1}{2}(M_{2}Fx)^{T}M_{1}(M_{2}Fx),\quad g_{2}(x)=\frac{1}{2}(Fx)^{T}M_{2}(Fx) \quad \text{and} \quad g(y)=\frac{1}{2}y^{T}M_{1}y$. 

\item $u_{3}(z)= \frac{1}{2}z^{T}(M_{1}^{-1} + M_{2})z, \quad u_{4}(w)= \frac{1}{2}w^{T}M_{2}^{-1}w $.

\item $\rho(x,y)=\sup_{z,w}[ b(x,y,z,w)-u_{3}(z)-u_{4}(w)]$.
\end{enumerate}
Consider the set: $$W=\Big\{ (x,y,z,w): x_{\perp}=A^{T}y, z=M_{1}(y+M_{2}Fx)\quad\text{and}\quad w=M_{2}(z+Fx)\Big\}$$
 We claim 
 \begin{equation}\label{eqn: potential inequality}
 b(x,y,z,w)- u_{1}(x)-u_{2}(y)-u_{3}(z)-u_{4}(w)\leq 0
 \end{equation}
  for all $(x,y,z,w)\in (\mathbb{R}^{2})^{4}$ and equality holds on $W$. For the inequality, it suffices to prove $ \rho(x,y)\leq u_{1}(x) + u_{2}(y)$.
\begin{flalign*}
\rho(x,y)& =x\cdot y + \sup_{z,w}[ y\cdot z+z\cdot w+w\cdot (Fx)-u_{3}(z)-u_{4}(w)]&\nonumber\\
    &=x\cdot y + \sup_{z}\big[ y\cdot z -u_{3}(z)+ \sup_{w}\big[(z+Fx)\cdot w -u_{4}(w)\big]\big]\nonumber\\
     &=x\cdot y + \sup_{z}\big[ y\cdot z -u_{3}(z)+ u_{4}^{*}(z+Fx)\big]
     \nonumber\\
      &=x\cdot y + \sup_{z}\big[ y\cdot z -u_{3}(z)+ \frac{1}{2}(z+Fx)^{T}M_{2}(z+Fx)\big]
      \qquad \qquad \text {by  \eqref{Lf}} \\
       &=x\cdot y + \sup_{z}\big[ y\cdot z - \frac{1}{2}z^{T}M_{1}^{-1}z+z^{T}M_{2}Fx\big]+ g_{2}(x)
       \nonumber\\ 
 &=x\cdot y +  \frac{1}{2}(y+M_{2}Fx)^{T}M_{1}(y+M_{2}Fx) + g_{2}(x)
\qquad\qquad\qquad\quad\text {by  \eqref{Lf}} \\
  &= x\cdot y + y^{T}M_{1}M_{2}Fx + g_{1}(x) + g_{2}(x) + g(y) 
   \nonumber\\
  &= y^{T}\cdot Ax + g_{1}(x) + g_{2}(x) + g(y)
 \nonumber\\
  &= y^{T}\cdot Ax_{\perp}  + g_{1}(x) + g_{2}(x) + g(y)
    \nonumber\\
    &= A^{T}y\cdot x_{\perp}  + g_{1}(x) + g_{2}(x) + g(y)
      \nonumber\\
    &\leq \frac{|A^{T}y|^{2}}{2}+ \frac{| x_{\perp} |^{2}}{2} + g_{1}(x) + g_{2}(x) + g(y)
     \qquad \qquad\quad\text{by the Cauchy-Schwarz Inequality}\\
  &= u_{1}(x) + u_{2}(y),
\end{flalign*}
with equality when $x_{\perp}=A^{T}y$. Hence, for any element $(x_{0},y_{0},z_{0},w_{0})$ in $W$, $\rho(x_{0},y_{0})=u_{1}(x_{0}) + u_{2}(y_{0})$. Furthermore, by tracing the cases of equality in the preceding string of inequalities, it is not hard to show that $(z_{0},w_{0})$ maximizes the map $(z,w)\mapsto  y_{0}\cdot z+z\cdot w+w\cdot (Fx_{0})-u_{3}(z)-u_{4}(w)$. Then $b(x_{0},y_{0},z_{0},w_{0})-u_{3}(z_{0})-u_{4}(w_{0})=\rho(x_{0},y_{0})$; that is $b(x_{0},y_{0},z_{0},w_{0})-u_{3}(z_{0})-u_{4}(w_{0})= u_{1}(x_{0}) + u_{2}(y_{0})$ on $W$, proving the claim.

Since $x_{\parallel}$ and $y$ can be chosen freely, $dim(W)= 3$. Then, if we take any probability measure $\mu$, concentrated on $W$ and absolutely continuous with respect to  3-dimensional Hausdorff measure, $spt(\mu)$ will not be concentrated on the graph of a measurable map. Now, take the projections of $\mu$ as marginals; that is, set $\mu_{1}=(\pi_{x})_{\sharp}\mu$, $\mu_{2}=(\pi_{y})_{\sharp}\mu$, $\mu_{3}=(\pi_{z})_{\sharp}\mu$ and $\mu_{4}=(\pi_{w})_{\sharp}\mu$). It is straightforward to show that $M_{1}$ and $M_{2}$ can be chosen so that the canonical projections $\pi_{x}$, $\pi_{y}$, $\pi_{z}$, $\pi_{w}$ from $W$ are surjective, and so
the marginals $\mu_{k}$ are absolutely continuous with respect to Lebesgue measure. Hence, inequality \eqref{eqn: potential inequality} together with the fact that equality holds $\mu$ almost everywhere implies that  $\mu$ is a solution to \eqref{KPb} which is not induced by a map.
\end{proof} 
\subsection{The condition $m=4$.}
In this subsection we show that the hypothesis on the numbers of variables in Theorem 3.1 is necessary. We will follow the ideas behind the proof of Proposition 4.2.1. 

In what follows, the presented variables are in $\mathbb{R}^{2}$. For a given $x\in\mathbb{R}^{2}$ its coordinates will be denoted by $x^{1}$, $x^{2}$.
\begin{proposition}
For $b(x_{1},\ldots, x_{m})=\sum_{k=1}^{m-1}x_{k}\cdot x_{k+1}+x_{m}\cdot x_{1}$, $m\geq5$, there are absolutely continuous marginals $\mu_{k}$ with respect to the Lebesgue measure $\mathcal{L}^{2}$,  such that a solution of \eqref{KPb} is not concentrated on a graph.
\end{proposition}
\begin{proof}
By part 2 of Theorem 2.2, $M =\begin{bmatrix}
-1& 3\\
0& -1
\end{bmatrix}\in \mathbb{\Re}^{3}$. Hence, we can choose positive definite matrices $M_{1}, M_{2}, M_{3} > 0$ such that $M=M_{1}M_{2}M_{3}$. For all $x_{1},\ldots, x_{m} $ define:
\begin{enumerate}[label=\roman*.]
\item  $u_{1}^{m}(x_{1})=\frac{3(x_{1}^{2})^{2}}{2}+g_{1}(x_{1}) + g_{2}^{m}(x_{1}), \quad $ where \quad $g_{1}(x_{1})=\frac{1}{2}(M_{3}x_{1})^{T}M_{2}(M_{3}x_{1})+\frac{1}{2}(M_{2}M_{3}x_{1})^{T}M_{1}(M_{2}M_{3}x_{1})$ \quad and \quad $g_{2}^{m}(x_{1})=\frac{1}{2}x_{1}^{T}M_{3}x_{1} +\frac{m-5}{2}|x_{1}|^{2}$ \quad for all $m\geq 5$.

\item $u_{2}(x_{2})= \frac{3(x_{2}^{1})^{2}}{2}+g(x_{2})$ \quad with \quad $g(x_{2})=\frac{1}{2}x_{2}^{T}M_{1}x_{2}$,\quad $u_{3}(x_{3})= \frac{1}{2}x_{3}^{T}(M_{1}^{-1} + M_{2})x_{3}$\quad and  \quad $u_{4}(x_{4})= \frac{1}{2}x_{4}^{T}(M_{2}^{-1}+M_{3})x_{4} $\quad for all $m\geq 5$.

\item $u_{5}(x_{5})=  
		\frac{1}{2}x_{5}^{T}M_{3}^{-1}x_{5}$ for  $m=5$.

\item For $m>5$,  $u_{k}(x_{k})=  \left\{
	\begin{array}{ll}
	\frac{1}{2}x_{5}^{T}(M_{3}^{-1}+I)x_{5} & \mbox{if } k=5\\
		|x_{k}|^{2}  & \mbox{if } 5<k<m \\
		\frac{1}{2}|x_{m}|^{2} & \mbox{if } k=m\\

	\end{array}
\right.$
\item $\rho^{m}(x_{1},x_{2})=\sup_{x_{3},..., x_{m}}\big [b(x_{1},\ldots, x_{m})-\sum_{k=3}^{m}u_{k}(x_{k})\big ]$ for all $m\geq 5$.
\end{enumerate}
Consider the set:

\begin{align*}
W&=\Big\{ (x_{1},\ldots, x_{m}): x_{1}^{2}=x_{2}^{1},\quad x_{3}=M_{1}(x_{2}+M_{2}M_{3}x_{1}), \quad x_{4}=M_{2}(x_{3}+M_{3}x_{1}),\\ 
& \qquad\qquad\qquad\quad x_{5}=M_{3}(x_{4}+x_{1})\quad \text{and}\quad x_{k}=x_{1}+x_{k-1},\quad\text{for} \quad k\geq 6\Big\}
\end{align*}

We claim $b(x_{1},\ldots, x_{m})-\sum_{k=3}^{m}u_{k}(x_{k})\leq u_{1}^{m}(x_{1}) + u_{2}(x_{2})$ for all $(x_{1},\ldots, x_{m})\in (\mathbb{R}^{2})^{m}$ and equality holds on $W$. For the inequality, it suffices to prove $\rho^{m}(x_{1},x_{2})\leq u_{1}^{m}(x_{1}) + u_{2}(x_{2})$ for all $m\geq 5$. We divide the proof of the claim into two cases:
\begin{enumerate}
\item For m=5
\begin{flalign*}
 \rho^{5}(x_{1},x_{2})& = x_{1}\cdot x_{2} + \sup_{x_{3},x_{4},x_{5}}[ x_{2}\cdot x_{3}+x_{3}\cdot x_{4}+x_{4}\cdot x_{5} +x_{5}\cdot x_{1}-u_{3}(x_{3})-u_{4}(x_{4})-u_{5}(x_{5})]\nonumber\\
    &=x_{1}\cdot x_{2} + \sup_{x_{3}}\Big[ x_{2}\cdot x_{3} -u_{3}(x_{3})+ \sup_{x_{4}}\big[x_{3}\cdot x_{4}-u_{4}(x_{4})\nonumber\\
     &\quad +\sup_{x_{5}}[(x_{4}+x_{1})x_{5}-u_{5}(x_{5})]\big]\Big]\nonumber\\
     &=x_{1}\cdot x_{2} + \sup_{x_{3}}\big[ x_{2}\cdot x_{3} -u_{3}(x_{3})+ \sup_{x_{4}}[x_{3}\cdot x_{4} -u_{4}(x_{4})+u_{5}^{*}(x_{4}+x_{1})]\big]\nonumber\\
      &=x_{1}\cdot x_{2} + \sup_{x_{3}}\big[ x_{2}\cdot x_{3} -u_{3}(x_{3})+ \sup_{x_{4}}[x_{3}\cdot x_{4} -u_{4}(x_{4})\\
      &\quad +\frac{1}{2}(x_{4}+x_{1})^{T}M_{3}(x_{4}+x_{1})]\big]\quad\qquad\qquad\qquad\qquad\qquad\quad\text {by  \eqref{Lf}} \\
      &=x_{1}\cdot x_{2} + \sup_{x_{3}}\big[ x_{2}\cdot x_{3} -u_{3}(x_{3})+ \sup_{x_{4}}[x_{3}\cdot x_{4} -\frac{1}{2}x_{4}^{T}M_{2}^{-1}x_{4}\\
      &\quad +x_{4}^{T}M_{3}x_{1}+g_{2}^{5}(x_{1})]\big]\\
      &=x_{1}\cdot x_{2} + \sup_{x_{3}}\big[ x_{2}\cdot x_{3} -u_{3}(x_{3})+\frac{1}{2}(x_{3}+M_{3}x_{1})^{T}M_{2}(x_{3}+M_{3}x_{1})\big]\\
      & \quad +g_{2}^{5}(x_{1}) \quad\qquad\qquad\qquad\qquad\qquad\qquad\qquad\qquad\qquad\qquad\text{by \eqref{Lf}}\\      
       &=x_{1}\cdot x_{2} + \sup_{x_{3}}\big[ x_{2}\cdot x_{3}-\frac{1}{2}x_{3}^{T}M_{1}^{-1}x_{3} + (M_{3}x_{1})^{T}M_{2}x_{3}\big] +g_{2}^{5}(x_{1})\\
       &\quad + \frac{1}{2}(M_{3}x_{1})^{T}M_{2}(M_{3}x_{1})\\
       &=x_{1}\cdot x_{2} + \sup_{x_{3}}\big[ x_{2}\cdot x_{3}-u_{3}(x_{3})+ \frac{1}{2}(x_{3}+M_{3}x_{1})^{T}M_{2}(x_{3}+M_{3}x_{1})\big] +g_{2}^{5}(x_{1})\nonumber\\
 &=x_{1}\cdot x_{2} +\frac{1}{2}(x_{2}+M_{2}M_{3}x_{1})^{T}M_{1}(x_{2}+M_{2}M_{3}x_{1})+g_{2}^{5}(x_{1})\nonumber\\
 &\quad +\frac{1}{2}(M_{3}x_{1})^{T}M_{2}(M_{3}x_{1})\quad\qquad\qquad\qquad\qquad\qquad\qquad\qquad\text {by  \eqref{Lf}}\nonumber\\
  &= x_{1}\cdot x_{2} + x_{2}^{T}Mx_{1} + g_{1}(x_{1}) + g_{2}^{5}(x_{1}) + g(x_{2}) 
   \nonumber\\
   &= 3x_{1}^{2} x_{2}^{1} + g_{1}(x_{1}) + g_{2}^{5}(x_{1}) + g(x_{2}) 
   \nonumber\\
   &\leq \frac{3(x_{1}^{2})^{2}}{2} + \frac{3(x_{2}^{1})^{2}}{2}+ g_{1}(x_{1}) + g_{2}^{5}(x_{1}) + g(x_{2})  \qquad\text{by the Cauchy-Schwarz Ineq.}\\
  &= u_{1}^{5}(x_{1}) + u_{2}(x_{2})
\end{flalign*}
\item The case $m\geq 6$ will be proved using induction. For $m=6$, note that:
$$x_{4}\cdot x_{5} -u_{5}(x_{5})+\sup_{x_{6}}[(x_{5}+x_{1})x_{6}-u_{6}(x_{6})]=(x_{1}+x_{4})x_{5}-\frac{1}{2}x_{5}^{T}M_{3}^{-1}x_{5}+\frac{1}{2}|x_{1}|^{2}$$
 and $$\sup_{x_{5}}[(x_{1}+x_{4})x_{5}-\frac{1}{2}x_{5}^{T}M_{3}^{-1}x_{5}+\frac{1}{2}|x_{1}|^{2}]= \frac{1}{2}(x_{4}+x_{1})^{T}M_{3}(x_{4}+x_{1})+\frac{1}{2}|x_{1}|^{2}.$$
Then
\begin{flalign*}
\rho^{6}(x_{1},x_{2})&=\rho^{5}(x_{1},x_{2}) + \frac{1}{2}|x_{1}|^{2}&\nonumber\\
&\leq u_{1}^{5}(x_{1}) + u_{2}(x_{2})+ \frac{1}{2}|x_{1}|^{2}\quad\qquad\qquad\qquad\qquad\qquad\qquad\text {(13)}&\nonumber\\
&=u_{1}^{6}(x_{1}) + u_{2}(x_{2})
\end{flalign*}

Assume the statement is true for $m-1$. Then

\begin{flalign*}
 \rho^{m}(x_{1},x_{2})& = \sup_{x_{3},..., x_{m}}\Big [\sum_{k=1}^{m-1}x_{k}\cdot x_{k+1}+x_{m}\cdot x_{1}-\sum_{k=3}^{m}u_{k}(x_{k})\Big ]&\nonumber\\
     &= \sup_{x_{3},..., x_{m-1}}\Big [\sum_{k=1}^{m-2}x_{k}\cdot x_{k+1}-\sum_{k=3}^{m-1}u_{k}(x_{k})+ \sup_{x_{m}}\big[(x_{1}+x_{m-1})x_{m}-u_{m}(x_{m})\big]\Big ]\nonumber\\
       &=\sup_{x_{3},..., x_{m-1}}\Big [\sum_{k=1}^{m-2}x_{k}\cdot x_{k+1}+\frac{|x_{1}+x_{m-1}|^{2}}{2}-\sum_{k=3}^{m-1}u_{k}(x_{k}) \Big ]\qquad\text{by  \eqref{Lf}}\nonumber\\  
 &=\sup_{x_{3},..., x_{m-1}}\Big [b(x_{1},\ldots, x_{m-1})-\sum_{k=3}^{m-2}u_{k}(x_{k})- \frac{|x_{m-1}|^{2}}{2}+\frac{|x_{1}|^{2}}{2}\Big ]\nonumber\\
  &= \rho^{m-1}(x_{1},x_{2}) +\frac{|x_{1}|^{2}}{2}\quad\qquad\qquad\qquad\qquad\qquad\qquad\qquad\text {(14)}
   \nonumber\\
   &\leq u_{1}^{m-1}(x_{1}) + u_{2}(x_{2}) +\frac{|x_{1}|^{2}}{2} \quad\qquad\qquad\qquad\qquad\text {by induction hypothesis}\\
  &= u_{1}^{m}(x_{1}) + u_{2}(x_{2}).
\end{flalign*}

If  $x_{1}^{2}=x_{2}^{1}$, we obtain $\rho^{5}(x_{1},x_{2})= u_{1}^{5}(x_{1}) + u_{2}(x_{2})$ and by (13), $\rho^{6}(x_{1},x_{2})=u_{1}^{6}(x_{1}) + u_{2}(x_{2})$. Furthermore, by (14) $\rho^{m}(x_{1},x_{2})=\rho^{m-1}(x_{1},x_{2}) +\frac{|x_{1}|^{2}}{2}$. Hence, using induction we can easily prove that $\rho^{m}(x_{1},x_{2})= u_{1}^{m}(x_{1}) + u_{2}(x_{2})$ for all $m\geq 5$, when $x_{1}^{2}=x_{2}^{1}$.\\
On the other hand, for any element $(x_{1}^{0}, \ldots, x_{m}^{0})$ in $W$, $(x_{3}^{0}, \ldots, x_{m}^{0})$ maximizes the map: $$(x_{3}, \ldots, x_{m})\mapsto x_{2}^{0}\cdot x_{3}+\sum_{k=3}^{m-1}x_{k}\cdot x_{k+1}+x_{m}\cdot x_{1}^{0}-\sum_{k=3}^{m}u_{k}(x_{k}).$$ Hence $b(x_{1}^{0}, \ldots, x_{m}^{0})-\sum_{k=3}^{m}u_{k}(x_{k}^{0})=\rho^{m}(x_{1}^{0},x_{2}^{0})$= $u_{1}^{m}(x_{1}^{0}) + u_{2}(x_{2}^{0})$. This proves the claim.

The rest of the proof is identical to the proof of Proposition \ref{prop:4 marginal counterexample}; since $x_1$ and $x_2^2$ can be chosen freely, $W$ is three dimensional, and the claim implies that any probability measure $\mu$ supported on $W$ is optimal for its marginals in \eqref{KPb}.
 \end{enumerate}

\end{proof}

\end{document}